\documentclass[reqno, 10pt]{amsart}

\usepackage[utf8]{inputenc}
\usepackage[T1]{fontenc}
\usepackage[english]{babel}

\usepackage{amssymb, amsfonts}

\usepackage{stmaryrd}

\usepackage{enumerate}
\usepackage{booktabs}

\usepackage{listings}

\usepackage{url}

\numberwithin{equation}{section}

\newtheorem{theoremcounter}{theoremcounter}[section]

\newtheorem{algorithm}[theoremcounter]{Algorithm}
\newtheorem{definition}[theoremcounter]{Definition}
\newtheorem{example}[theoremcounter]{Example}
\newtheorem{lemma}[theoremcounter]{Lemma}
\newtheorem{remark}[theoremcounter]{Remark}

\newtheorem{theorem}[theoremcounter]{Theorem}

\newcommand{\tit}{\itshape}

\newcommand{\nbd}{\nobreakdash-\hspace{0pt}}

\renewcommand{\frak}{\ensuremath{\mathfrak}}
\newcommand{\cal}{\ensuremath{\mathcal}}
\newcommand{\bboard}{\ensuremath{\mathbb}}

\newcommand{\frakA}{\ensuremath{\frak{A}}}
\newcommand{\frakB}{\ensuremath{\frak{B}}}
\newcommand{\frakC}{\ensuremath{\frak{C}}}
\newcommand{\frakD}{\ensuremath{\frak{D}}}
\newcommand{\frakE}{\ensuremath{\frak{E}}}

\newcommand{\cD}{\ensuremath{\cal{D}}}
\newcommand{\cE}{\ensuremath{\cal{E}}}
\newcommand{\cF}{\ensuremath{\cal{F}}}

\newcommand{\cO}{\ensuremath{\cal{O}}}

\newcommand{\cQ}{\ensuremath{\cal{Q}}}

\newcommand{\bbP}{\ensuremath{\bboard P}}

\newcommand{\NN}{\ensuremath{\mathbb{N}}}
\newcommand{\ZZ}{\ensuremath{\mathbb{Z}}}
\newcommand{\QQ}{\ensuremath{\mathbb{Q}}}
\newcommand{\RR}{\ensuremath{\mathbb{R}}}
\newcommand{\CC}{\ensuremath{\mathbb{C}}}

\renewcommand{\Im}{\ensuremath{\mathop{\mathfrak{Im}}}}

\newcommand{\Mat}[2]{\ensuremath{\mathrm{M}_{#1}(#2)}}

\newcommand{\GL}[1]{\ensuremath{\mathrm{GL}_{#1}}}
\newcommand{\SL}[1]{\ensuremath{\mathrm{SL}_{#1}}}

\newcommand{\Orth}[2]{\ensuremath{\mathrm{O}_{#1,#2}}}

\newcommand{\tr}{\ensuremath{\mathrm{tr}}}

\newcommand{\slashdiv}{\ensuremath{\mathop{/}}}

\newcommand{\HS}{\mathbb{H}}

\newcommand{\td}{\tilde}

\newcommand{\ov}{\overline}

\newcommand{\Mp}[1]{\ensuremath{\mathrm{Mp}_{#1}}}

\newcommand{\disc}{\ensuremath{\mathrm{disc}}}

\newcommand{\OrthL}{\ensuremath{\mathrm{O}(L)}}
\renewcommand{\pmod}[1]{\ensuremath{\;(\mathrm{mod}\, #1)}}

\begin{document}
\title{Computing Borcherds products}

\author{Dominic Gehre}
\address{Lehrstuhl A f\"ur Mathematik\\ RWTH Aachen University\\ Templergraben 55\\\newline D\nbd 52056~Aachen, Germany}
\email{dominic.gehre@matha.rwth-aachen.de}

\author{Judith Kreuzer}
\address{Lehrstuhl A f\"ur Mathematik\\ RWTH Aachen University\\ Templergraben 55\\\newline D-52056~Aachen, Germany}
\email{judith.kreuzer@matha.rwth-aachen.de}

\author{Martin Raum}
\address{ETH Mathematics Dept.\\ R\"amistra\ss e~101\\ CH-8092, Z\"urich, Switzerland}
\email{martin.raum@math.ethz.ch}
\urladdr{http://www.raum-brothers.eu/martin/}

\thanks{The second author holds a scholarship of the Graduate School ``Experimental and Constructive Algebra'' of the RWTH Aachen University.  The third author held a scholarship of the Max Planck Society.  He currently holds an ETH Postdoctoral Fellowship, which is cofunded by Marie Curie actions COFUND}

\date{}

\begin{abstract}
We present an algorithm for computing Borcherds products, which has polynomial runtime.  It deals efficiently with the bounds on Fourier expansion indices originating in Weyl chambers.  Naive multiplication has exponential runtime due to inefficient handling of these bounds.  An implementation of the new algorithm shows that it is also much faster in practice.
\end{abstract}

\maketitle

\section{Introduction}

Product expansions are a useful construction in the world of automorphic forms.  For example, writing $\tau$ and $\tau'$ for two variables in the Poincar\'e upper half plane, we can express the famous $j$-invariant as follows:
\begin{gather*}
  j(\tau) - j(\tau')
=
  \big(e^{-2 \pi i \tau'} - e^{-2 \pi i \tau} \big)
  \prod_{m,n = 1}^\infty
  \big(1 - e^{2 \pi i (m \tau + n \tau')}\big)^{c_{m n}}
\end{gather*}
for certain integral coefficients $c_{m n}$.  This formula holds for sufficiently large imaginary parts of $\tau$ and $\tau'$.  Several facts can be deduced from it.  For instance, one sees immediately that close to infinity, the $j$-function is one-to-one on the fundamental domain.

It turns out that product expansions exist for a much larger class of automorphic forms, namely orthogonal modular forms for $\Orth{2}{n}(\RR)$.  Such product expansions allow for a direct investigation of all zeros of such an automorphic form, regardless of the limited domain of convergence they have.

The convergence of these more general product expansions remained an open problem for a long time.  When Borcherds presented his work on product expansions for orthogonal groups~\cite{Bo98} (see also \cite{Bo95, Ko97, HM98}), he started a revolution.  Product expansions, which at that time played a role in the classical theory of elliptic modular forms and Jacobi forms only, became important to other areas in mathematics.  One branch of research aimed at Kac-Moody algebras which could be constructed by means of lattice theory~\cite{Gr95, GN98, Sch04, Sch08}.  As a natural generalization of finite dimensional Lie algebras, they attracted the physicists' attention.  On the automorphic forms side, research in this area focuses on orthogonal groups for rather big lattices.  A second, equally popular branch of research originating in Borcherds' work is concerned with the structure of divisor class groups and Chern classes~\cite{Bo99, Bo00, Br02, CG08, BO10}.  It is founded on the fact that one can easily read off the divisor of an automorphic form given by a product expansion.  Lattices with particular nice structure, which are typically of moderate size, are in the focus of mathematicians working in this direction.  The Fourier expansion of Borcherds products is a particular aspect that is interesting in this context.  Among other things, it can be used to understand the structure of graded rings of modular forms~\cite{DK04, Kr05, Ma10}, or to check for relations like the ones coming from Saito\nbd Kurokawa-Maa\ss-Gritsenko lifts~\cite{HM10}.  Most importantly, the multiplicative version of the  degeneracy of BPS dyons, occurring in string theory, are Borcherds product of this type (see, for example,~\cite{HM98, JS05, DMZ11}).  The last named author will return to this subject in a sequel to this paper.

\vspace{1ex}

The purpose of this work is to provide an algorithm for explicit computations with Borcherds products that runs in polynomial time.  In particular, such an algorithm needs to build up on efficient bounds on the Fourier expansion indices that arise from positivity conditions related to Weyl chambers.  More precisely, the (infinite) products that we have to deal with essentially have the form (see~(\ref{eq:borcherds_psif}) for a precise formula) 
\begin{multline*}
  \prod_{[a,b,c] \succ 0} \big(\, 1 - e\big(a \tau + c \tau' - (b, z)_{L_0}\big) \, \big)^{f([a,b,c])}
\\[4pt]
=
  \sum_{[a,b,c]} g([a,b,c]) e\big(a \tau + c \tau' - (b, z)_{L_0}\big)
\text{,}
\end{multline*}
where $e(x) = \exp(2 \pi i \, x)$,  $[a,b,c]$ are vectors in a lattice of signature~$(1, n - 1)$, and $f([a,b,c])$ are certain coefficients of an elliptic modular form.  By definition, we have $[a, b, c] \succ 0$ if $[a,b,c]$ has positive scalar product with all elements in a certain subset, called a Weyl chamber, of that lattice.  One typically wants to evaluate the above product up to a certain bound.  That is, for given $B > 0$, one aims at computing all $g([a,b,c])$ with $0 \le a, c < B$.    Although the positivity condition $[a, b, c] \succ 0$ is natural and can be described in simple terms, its influence on the product as a whole is involved.  For example, a priori there are infinitely many $f([a,b,c])$ that contribute to $g([1, 0, 1])$.  Further considerations reveal that only finitely many contribute, and it is an important matter to provide good corresponding bounds.

When it comes to explicitly evaluating Borcherds products, the state of the art is~\cite{Ma10}, dealing with Hilbert modular forms.  Remark~5.1 and Lemma~5.9~(5) of~\cite{Ma10} contain the only considerations concerning the resulting bounds on those $[a,b,c]$ for which the powers of $( 1 - \exp (a \tau + c \tau' - (b, z)_{L_0}) )$ need to be computed.  Moreover, in the case of Hilbert modular forms the component $b$ does not appear, something that simplifies the corresponding estimates.  In the case of rank $2$ Lie groups $\Orth{2}{n}(\RR)$, $n \ge 3$ no similar estimate can be found in the literature.  Some computations of Borcherds products in this setting can be found in the physics literature, see, e.g.,~\cite{Kr10}.  Without bounds for the contributing indices, however, a rigorous evaluation of Borcherds products is not possible;  and with bad bounds it is inefficient.  I.e., a naive evaluation of Borcherds products in a strict sense is not even possible.

To provide efficient bounds for all orthogonal modular forms, we employ two ideas:  to make the problem as linear as possible and to decompose the resulting expression into parts that behave differently.  In order to achieve the first objective, we first apply the logarithm and then the exponential to Borcherds' original formula.  We then evaluate the logarithmic term first.  This avoids the repeated multiplication of polynomials, and it allows for a more efficient estimate of the indices of the non-vanishing terms in the Fourier expansion.  To achieve the second objective, we analyze the structure of Fourier indices that are positive with respect to a fixed Weyl chamber (see \eqref{eq:frakA_part} to \eqref{eq:frakE_part}).  This allows us to significantly reduce the number of terms taken into consideration.

In order to illustrate our algorithm, we provide an implementation in Sage~\cite{sage, singular} of the case of Hermitian modular forms over $\QQ \big(\sqrt{-3}\big)$ .  We compare it to an implementation based on naive multiplication (mainly using FLINT~\cite{flint}, written in C).  Even though our implementation is not written in a compiled language, it is faster by several orders of magnitude.  Finally, we also discuss how to provide input data for our implementation.

\vspace{1ex}

In Section~\ref{sec:preliminaries}, we review the basic theory of orthogonal modular forms and Borcherds products.  Section~\ref{sec:algorithm} deals with the algorithm to compute Fourier expansions of Borcherds products.  The implementation for Hermitian modular forms is discussed in Section~\ref{sec:hermitianborcherdsproducts}.  In particular, this section contains a comparison of runtimes.


{\tit Acknowledgment: The authors would like to thank the referee for an extraordinarily thorough and constructive critique of the original manuscript.}
\section{Preliminaries}
\label{sec:preliminaries}

\subsection{Orthogonal modular forms}
\label{ssec:orthogonalmodularforms}

Throughout the paper, we fix an even lattice $L$ of signature $(2, n)$ with $n \ge 3$.  The quadratic form attached to $L$ is denoted by $q_L$, while we write $(\mu,\lambda)_L$ for the bilinear form $q_L(\mu + \lambda) - q_L(\mu) - q_L(\lambda)$.  There is a natural extension of $q_L$ to $L \otimes \QQ$ that we will denote by the same letter.  The dual lattice of $L$ is
\begin{align*}
  L^\#
&=
  \big\{\lambda \in L \otimes \QQ : (\mu, \lambda)_L \in \ZZ \text{ for all } \mu \in L\big\}
\text{.}
\end{align*}
Clearly, $L \subseteq L^\#$.  Thus we can construct the $\ZZ$-module $L^\# / L$, that is equipped with the quadratic form $\ov{q}_L \equiv q_L \pmod{1}$ taking values in $\QQ / \ZZ$.  The pair $(L^\# / L,\,\ov{q}_L)$ is a finite quadratic module of order $|L|$.  Given a basis for $L$, we find that $|L|$ equals the determinant of the associated Gram matrix.  We call $(L^\# / L,\,\ov{q}_L)$ the discriminant module of $L$, which we denote by $\disc\, L$.

We write $\OrthL$ for the orthogonal group associated to $L$.  It is the group of all $\ZZ$-linear bijections of $L$ that leave the quadratic form $q_L$ invariant.  Any such transformation acts also on $L^\#$ and $\disc\, L$.  The discriminant kernel $\OrthL[\disc\, L]$ consists of all transformations in~$\OrthL$ fixing $\disc\, L$ elementwise.

We consider a fixed connected component $\cD_e^+$ of
\begin{align*}
  \cD_e
:=
  \big\{\lambda \in L \otimes \CC \,:\, q_L(\lambda) = 0 \wedge (\lambda,\overline \lambda)_L > 0\big\}
\text{,}
\end{align*}
where $\ov{\lambda}$ is the complex conjugate in the second tensor component of $\lambda$.  The projectivisation $\cD^+ = \bbP(\cD_e^+) = \big( \cD_e^+ \setminus \{ 0 \} \big) \slashdiv \CC$, which is a Hermitian locally symmetric domain of type IV~\cite{Hel78}, is the natural domain of definition for orthogonal modular forms.  We define these using the $-k$-homogeneous pullback to $\cD_e \setminus \{ 0 \}$.  A function $f \,:\, \cD_e^+ \setminus \{ 0 \} \rightarrow \CC$ is called $-k$-homogeneous if $f(s \lambda) = s^{-k} f(\lambda)$ for all $s \in \CC \setminus \{ 0 \}$ and all $\lambda \in \cD_e^+ \setminus \{ 0 \}$.
\begin{definition}
An orthogonal modular form of weight $k$ for a finite index subgroup $\Gamma \subset O(L)$ is a holomorphic, $-k$-homogeneous function $f \,:\, \cD_e \rightarrow \CC$ that is invariant under~$\Gamma$: $f \circ \gamma = f$ for all $\gamma \in \Gamma$.
\end{definition}
\noindent After fixing a section $\cD^+ \hookrightarrow \cD_e$, we identify any orthogonal modular form with its restriction to $\cD^+$.

We are interested in Borcherds products that are associated to lattices that contain two hyperbolic planes.  Hence, throughout the paper, we assume that $L = U \oplus U \oplus (-1)L_0$, where $U$ is the unimodular hyperbolic lattice of rank $2$ and $L_0$ is a fixed even positive definite lattice.  We write $q_{L_0}$ and $(\cdot\,,\,\cdot)_{L_0}$ for the quadratic and bilinear form associated to $L_0$.  Typically, $L_0$ is some order in a composition algebra~$A$ together with the quadratic form attached to $A$ restricted to this order.  Note that $\disc\, L = \disc\, (-1)L_0$.  There is a canonical section $\cD^+ \hookrightarrow \cD_e$ whose image has elements $\big( 1, q_{L_0}(z) - \tau \tau', \tau, \tau', z \big)$, $\tau, \tau' \in \CC$, $z \in L_0 \otimes \CC$.  Here, the first and the second entry, and the third and fourth entry are coordinates with respect to a standard basis of the first and second copy of $U$, respectively.  This element is always isotropic, since
\begin{gather*}
  1 \cdot \big( q_{L_0}(z) - \tau \tau'\big)
  + \tau \cdot \tau'
  + (-1) q_{L_0} (z)
=
  0
\text{.}
\end{gather*}
It lies in $\cD_e$ if
\begin{gather*}
  q_L\big( \Im(1, q_{L_0}(z) - \tau \tau', \tau, \tau', z) \big)
>
  0
\text{,}
\end{gather*}
where $\Im$ is applied to the second tensor component.

We shall show in a moment that in this setting, any orthogonal modular form has a Fourier expansion
\begin{gather}
\label{eq:orthogonal_fourier_expansion}
  \sum \alpha([a,b,c]) \, e\big( a \tau + c \tau' - (b, z)_{L_0} \big)
\text{,}
\end{gather}
where $e(x) := \exp(2 \pi i \, x)$.  The sum ranges over triples $[a,b,c]$ with $a, c \in \QQ$ and $b \in L_0 \otimes \QQ$, which correspond to vectors $(0, 0, a, c, b) \in L \otimes \QQ$.  We write $\cQ_\QQ$ for the additive group of such triples.  The submonoid of integral indices $[a,b,c] \in \cQ_{\QQ}$ that satisfy $a,c \in \ZZ$ and $b \in L_0^\#$, is denoted by $\cQ$.  All Fourier expansions that we will deal with satisfy $\alpha([a,b,c]) = 0$, if $[a,b,c] \not \in \cQ$.

We call an index positive definite or semi-definite if
\begin{gather*}
  \disc([a,b,c])
:=
  q_l\big((0, 0, a, c, b)\big)
=
  a c - q_{L_0}(b) > 0
\quad\text{and}\quad
  a > 0
\text{,}
\end{gather*}
or $\disc([a,b,c]) \ge 0$ and $a, c \ge 0$.  We write $[a,b,c] > 0$ if $[a,b,c]$ is positive definite, and $[a,b,c] \ge 0$ if it is positive semi-definite.  The monoid of integral positive semi-definite indices is denoted by $\cQ^+$.  

Since the authors are not aware of any publicly available, explicit discussion of the above Fourier expansion~(\ref{eq:orthogonal_fourier_expansion}), we briefly prove it.  Fix an orthogonal modular form~$\Phi$ for a finite index subgroup $\Gamma \subset \OrthL$, and view it as a function of $\tau$, $\tau'$, and $z$.  Set $u = (0, 1, 0, 0, 0) \in L \otimes \QQ$.  Given an element of $[a', b', c'] = (0, 0, a', c', b')$ of $U \oplus (-1) L_0 \subset L$, we can define Eichler transformations~$E\big( u, [a', b', c'] \big) \in \OrthL$~(see~\cite{FH00}, Rem.~4.2).  By definition, we have
\begin{align*}
& {}
  E\big( u, [a', b', c'] \big) \, \big( 1, q_{L_0}(z) - \tau \tau', \tau, \tau', z \big)
\\
= & {}
  \big(1, q_{L_0}(z) - \tau \tau', \tau, \tau', z \big)
  - [a', b', c]
  + \big( \big( c' \tau + a' \tau' - (b', z)_{L_0} \big) - \disc([a', b', c']) \big) u
\\
= & {}
  \big( 1, q_{L_0}(z) - \tau \tau' + c' \tau + a' \tau' - (b', z)_{L_0} - \disc([a', b', c']) , \tau - a', \tau' - c', z - b' \big)
\text{.}
\end{align*}
Consequently, invariance under $E\big( u, [a', b', c'] \big)$ implies that $\Phi(\tau, \tau', z) = \Phi(\tau + a', \tau' + c', z + b')$.  If $\Gamma = \OrthL$, then $\Phi$ is invariant under all $E\big( u, [a', b', c'] \big)$, $[a', b', c'] \in U \oplus (-1) L_0$.  It hence has a Fourier expansion~(\ref{eq:orthogonal_fourier_expansion}) with $[a, b, c]$ running through~$\cQ$.  In general, since $\Gamma$ has finite index in $\OrthL$, $\Phi$ is invariant under $E\big( m\cdot u, [m a', m b', m c'] \big)$ for some positive integers~$m$ and all $[a', b', c'] \in U \oplus (-1) L_0$.  In this case, the Fourier expansion of $\Phi$ is indexed by elements of $\frac{1}{m} \cQ \subset \cQ_\QQ$.

The discriminant kernel of the second and third component of $L$,
\begin{gather*}
  M
:=
  O(U \oplus (-1) L_0)[\disc\, (-1)L_0] \,\cap\, \Gamma \subset O(L)
\text{,}
\end{gather*}
gives rise to symmetries of the Fourier expansion. More precisely, we have
\begin{gather*}
  \alpha([a,b,c])
=
  \chi_k(m)\, \alpha(m [a,b,c])
\end{gather*}
for all $m \in M$ and a character $\chi_k$ of $M$ that depends only on the weight of an orthogonal modular form.

\begin{remark}
For an arbitrary lattice $\widetilde{L}$ of signature~$(2, n)$ with $n \ge 5$, it is always possible to find a sublattice $\widetilde{L}' \subseteq \widetilde{L}$ of finite index that splits: $\widetilde{L}' \simeq u_1 U \oplus u_2 U \oplus (-1) \widetilde{L}'_0$ with positive constants $u_1$ and $u_2$ (see, \cite{Bo98}, Section~8; This follows from the fact that any such lattice contains a non-zero isotropic vector, see \cite{Si51}, Satz~2 and the classification of local lattices, which is explained in~\cite{OM00}).  Since the orthogonal groups of $\widetilde{L}$ and $\widetilde{L}'$ are commensurable, the considerations in this paper apply to arbitrary Borcherds products for lattices of signature~$(2, n)$ with $n \ge 5$.
\end{remark}

\subsection{Borcherds products}
\label{ssec:borcherdsproducts}

We review the construction of Borcherds products, which dates back to~\cite{Bo95} and~\cite{Bo98}.  A more accessible discussion, valid in some interesting special cases, can be found in~\cite{Gr95, GN98}.  Given an elliptic vector valued modular form, we construct an orthogonal modular form.

Let $\HS := \{\tau \in \CC \,:\, \Im(\tau) > 0\}$ denote the Poincaré upper half-plane.  Since we will deal with half-integral weights,  we first have to define the metaplectic cover $\Mp{2}(\ZZ)$ of $\SL{2}(\ZZ)$.  It is the preimage of $\SL{2}(\ZZ)$ in $\Mp{2}(\RR)$, the connected double cover of $\SL{2}(\RR)$.  Recall that if $N$ is a positive integer then the principal congruence subgroup of level~$N$, denoted by~$\Gamma(N)$, is defined as the kernel of the natural projection map $\SL{2}(\ZZ) \rightarrow \SL{2}(\ZZ \slashdiv N \ZZ)$.  The principle congruence subgroup of level~$N$ of $\Mp{2}(\ZZ)$ is defined as the preimage of $\Gamma(N)$.

Write $g = \left(\begin{smallmatrix} a & b \\ c & d \end{smallmatrix}\right)$ for a typical element of $\SL{2}(\RR)$.  The elements of $\Mp{2}(\RR)$ can be written $\big(g,\, \tau \mapsto \sqrt{c \tau + d}\big)$, where the first component is an element of $\SL{2}(\RR)$ and the second is a holomorphic function on $\HS$.  Since there are two branches of the square root, this yields indeed a double cover of $\SL{2}(\RR)$.

Given a representation $(\rho, V_\rho)$ of $\Mp{2}(\ZZ)$, $k \in \tfrac{1}{2}\ZZ$, $F \,:\, \HS \rightarrow V_\rho$, and $(g, \omega) \in \Mp{2}(\ZZ)$, we define
\begin{gather*}
  \big( F\big|_{k,\rho} \, (g, \omega) \big)\, (\tau)
:=
  \omega(\tau)^{-2 k} \, \rho\big((g, \omega)\big)^{-1} \,
  F\Big( \frac{a \tau + b}{c \tau + d} \Big)
\text{.}
\end{gather*}
\begin{definition}
Let $(\rho, V_\rho)$ be a finite-dimensional representation of $\Mp{2}(\ZZ)$.  A weakly holomorphic vector valued modular form of type $\rho$ and weight $k \in \tfrac{1}{2}\ZZ$ is a holomorphic function $F \,:\, \HS \rightarrow V_\rho$ such that the following conditions are satisfied:
\begin{enumerate}
\item For all ${\td g} \in \Mp{2}(\ZZ)$ we have $F|_{k,\rho} \, {\td g} = F$.
\item The Fourier expansion of $F$ has the form
 \begin{gather*}
  \sum_{-\infty \ll m \in \QQ} f(m) \, e(m \tau),
 \quad
  f(m) \in V_\rho
 \text{.}
 \end{gather*}
\end{enumerate}
\end{definition}
\begin{remark}
If the weight~$k$ is integral and $\rho$ is trivial on some principal congruence subgroup, 
the map
\begin{gather*}
  \Mp{2}(\ZZ) \rightarrow \GL{}(V_\rho)
\text{,}\quad
  (g, \omega)
  \mapsto
  \omega(\tau)^{-2k} \rho\big( (g, \omega) \big)
\end{gather*}
factors through~$\SL{2}(\ZZ)$.
\end{remark}

The case that is most relevant to us is $\rho = \rho_L$, where $\rho_L$ is the Weil representation associated to the finite quadratic module $\disc\, L$.  The space $V_\rho$ has a canonical basis indexed by $\disc\, L$.  We refer the reader to \cite{Bo98, Sk08} for a definition and more details.  Suppose that $F$ is an elliptic modular form of type $\rho_L$.  Write
\begin{gather}
\label{eq:vvell_fourierexpansion}
  F_{\lambda}(\tau) = \sum_{-\infty \ll m} f( \ov{\lambda}, m)\, e( m \tau)
\end{gather}
for the Fourier expansion of the components $F_{\ov{\lambda}}$ of $F$, $\ov{\lambda} \in \disc\, L$.  We often write $f(\lambda, m)$ with $\lambda \in L^\#$, meaning $f(\ov{\lambda}, m)$ with $\ov{\lambda} \equiv \lambda \pmod{L}$.

Given a fixed basis $w_1,\ldots, w_{n-2}$ of $L_0 \otimes \RR$, we write $b \succ 0$ if there is $1 \le j \le n - 2$ such that $(b, w_j)_{L_0} > 0$ and $(b, w_{j'})_{L_0} = 0$ for all $1 \le j' < j$.  A triple $[a, b, c]$ is called positive, $[a, b, c] \succ 0$, if $c > 0$, or $c = 0$ and $a > 0$, or $a = c = 0$ and $b \succ 0$.  This definition is motivated by Weyl chambers discussed in \cite{Bo98}, Section~6.  We can find a Weyl chamber $W$ for $U \oplus (-1) L_0 \subset L$ such that $([a, b, c], W) = \big( (0, 0, a, c, b),\, W \big) > 0$ if and only if $[a, b, c] \succ 0$, where all $w_j$ are rational vectors.

Borcherds also defined a Weyl vector attached to $F$ and $W$.  It gives rise to the exponential factor~$W_F$ defined below.  In our setting (see also \cite{De01}, Satz~5.4), we have:
\begin{align*}
  a_W
&:=
  \frac{1}{24} \sum_{b \in L_0^\#} f\big(b, -q_{L_0}(b)\big)
\text{,}
&
  b_W
&:= 
  - \frac{1}{2} \sum_{ 0 \prec b \in L_0^\#} f\big(b, -q_{L_0}(b)\big) \cdot b
\text{,}
\\[4pt]
  c_W
&:=
 a_W - \sum_{n}^\infty \sigma_1(n) \sum_{b \in L_0^\#} f\big(b, - n - q_{L_0}(b)\big)
\text{,}
\\[4pt]
  W_F
&:=
  e\big(a_W \tau + c_W \tau' - (b_W,z)_{L_0}\big)
\text{.}
\end{align*}

With this notation at hand, we reformulate~\cite{Bo98}, Theorem~13.3.  Note that $\Psi_F$ is only well-defined up to a complex multiplicative scalar of norm~$1$.  In addition, we rescale $\Psi_F$ by a real factor which can be explicitly computed in terms of the $f(\lambda, 0)$.
\begin{theorem}
\label{thm:borcherds_products}
Fix a weakly holomorphic vector valued elliptic modular form $F$ of type $\rho_L$ and weight $(2 - n)/2$  with Fourier expansion~\eqref{eq:vvell_fourierexpansion} and integral $f(\lambda, m)$.  Then the following product converges for $(\tau, \tau', z) \in \RR^n + i W \subset \big(U \oplus (-1) L_0 \big) \otimes \CC$ and is an orthogonal modular form of weight $f(0, 0) \slashdiv 2$ for $\OrthL[\disc\, L]$:
\begin{gather}
\label{eq:borcherds_psif}
  \Psi_F
:=
  W_F
  \prod_{[a,b,c] \succ 0} \big(\, 1 - e\big(a \tau + c \tau' - (b, z)_{L_0}\big) \, \big)^{f(b, \disc([a,b,c]))}
\text{.}
\end{gather}
\end{theorem}

By definition of~$W$, we have
\begin{gather*}
  \big| e\big( a \tau + c \tau' - (b, z)_{L_0} \big) \big| < 1
\end{gather*}
for all $[a, b, c] \succ 0$ and $(\tau, \tau', z) \in \RR^n + iW$.  Hence we can define a holomorphic logarithm of $\Psi_F$ on $\RR^n + i W$, which is non-zero there.
\begin{align}
\nonumber
{} &
  \Phi_F
=
  \log \Psi_F
\\
\label{eq:borcherds_phif}
= {} &
  \log W_F + \!\!\sum_{[a,b,c] \succ 0} f\big( b, \disc([a,b,c]) \big) \,
                                \log \big(\, 1 - e\big(a \tau + c \tau' - (b, z)_{L_0}\big)\,\big)
\text{.}
\end{align}
This logarithm will be crucial to our approach.


\section{An algorithm for Fourier expansions of Borcherds products}
\label{sec:algorithm}

We write $\ov{\QQ}\llbracket \cQ \rrbracket$ for the direct product of $1$-dimensional $\ov{\QQ}$ vector spaces indexed by $\cQ$.  Any element will be written as a possibly infinite sum of multiples of $e^{[a,b,c]}$ for $[a,b,c] \in \cQ$.  Here, $e^{[a,b,c]}$ is simply a convenient notation for basis vectors of this space.  We say that $[a, b, c]$ occurs non-trivially in an element of ${\ov \QQ}\llbracket \cQ \rrbracket$ if the coefficient of $e^{[a, b, c]}$ is nonzero.

Fix a vector valued elliptic modular form $F$ of type $\rho_L$ and weight $(2 - n) / 2$. Define
\begin{gather}
  {\widetilde \Phi}_F
:=
  - \!\!\! 
  \sum_{[a,b,c] \succ 0} f\big( b, \disc([a,b,c]) \big) \,
  \sum_{m = 1}^\infty \frac{e^{m [a,b,c]}}{m}
\in
  \ov{\QQ}\llbracket \cQ \rrbracket
\text{.}
\end{gather}
This is well-defined, since only finitely many term contribute to $[a',b',c']$-th component of ${\widetilde \Phi}_F$ for each~$[a',b',c'] \in \cQ$.  For example, if $c' > 0$, then only those $[a, b, c] \succ 0$ contribute that satisfy $c \le c'$ and $\disc([a, b, c]) > d$ for some $d$ that depends on $f$.  A reasoning similar to this will be used while proving Lemma~\ref{la:phif_powers}.

We can map this formal expansion to a function on $\RR^n + i W$, where as stated before, we have
\begin{gather*}
  \big| e\big( a \tau + c \tau' - (b, z)_{L_0} \big) \big| < 1
\end{gather*}
for all $[a, b, c] \succ 0$.  Clearly, ${\widetilde \Phi}_F$ is mapped to $\Phi_F - \log W_F$ as $e^{[a,b,c]}$ is mapped to $e\big( a \tau + c \tau' - (b, z)_{L_0} \big)$.

Note that $\ov{\QQ}\llbracket \cQ \rrbracket$ does not carry an algebra structure, since it contains $\ov{\QQ}\llbracket \{ [a, 0, 0] \,:\, a \in \ZZ \}\rrbracket$, which is isomorphic to the space of two-ended formal power series over $\ov{\QQ}$.  However, we can define powers of ${\widetilde \Phi}_F$.  Write ${\widetilde \phi}_F ([a,b,c])$ for the coefficient of $e^{[a,b,c]}$ in~${\widetilde \Phi}_F$.  We set
\begin{gather}
  {\widetilde \Phi}_F^l
:=
  \sum_{[a,b,c]} \bigg( \sum_{\sum_{i=1}^l [a_i,b_i,c_i] = [a,b,c]} \prod {\widetilde \phi}_F([a_i,b_i,c_i]) \bigg) \, e^{[a,b,c]}
\text{,}
\end{gather}
where the $[a_i,b_i,c_i]$'s are elements of $\cQ$.  By the next lemma, the inner sum is finite.
\begin{lemma}
\label{la:phif_powers}
Given $l \in \NN$, the set
\begin{gather*}
  \Big\{
   \big( [a_i,b_i,c_i] \big)_{i=1,\ldots,l} \in \cQ^l
  \,:\,
   \sum_{i=1}^l \, [a_i,b_i,c_i]
  =
   [a,b,c]
  ,\,
   {\widetilde \phi}_F ([a_i,b_i,c_i]) \ne 0
  \text{ for all } i
  \Big\}
\end{gather*}
is finite.
\end{lemma}
\begin{proof}
For reasons of symmetry, it suffices to prove that, given arbitrary $[a_i, b_i, c_i]$'s with $i \ne 1$, there are only finitely many $[a_1, b_1, c_1]$ that possibly contribute.  Assume that the tuple $\big( [a_i,b_i,c_i] \big)_{i=1,\ldots,n} \in \cQ^l$ is an element of the above set.

Since $F$ is a weakly holomorphic modular form, there is a lower bound $d$ such that $f(b, D) = 0$, whenever $D \le d$.  By definition of the relation $\succ$ and of ${\widetilde \Phi}_F$, only indices $[a_1,b_1,c_1]$ satisfying $c_1 \ge 0$ contribute.  In particular, we have $c_1 \le c$ for those $[a_1, b_1, c_1]$ that contribute.  Consider those indices with $c_1 > 0$.  If $\phi_F([a_1, b_1, c_1]) \ne 0$, we can write $[a_1, b_1, c_1] = m\, [\td a_1, \td b_1, \td c_1]$ with $m \in \NN$ and $\disc( [\td a_1, \td b_1, \td c_1] ) > d$.  This gives rise to the inequality $\disc( [a_1,b_1,c_1] ) > m^2 d$.  Because $m \le c_1 \le c$, we obtain the weaker inequality $\disc( [a_1,b_1,c_1] ) > c^2 d$.  From this, we deduce a lower bound $A_{\rm l}$ on $a_1$ for those $[a_1, b_1, c_1]$ with $c_1 > 0$ that contribute.  We can assume that $A_{\rm l} \le 0$, so that $a_1 \ge A_{\rm l}$ also holds in the case of $c_1 = 0$, as is explained later.  For reasons of symmetry $A_{\rm l}$ is then a lower bound on all other $a_i$'s, and there are exactly $(l - 1)$ of them.  From this we deduce the following upper bound
\begin{gather*}
  a_1
=
  a - \sum_{2 \le i \le n - 2} a_i
\le
  a - (l - 1) A_{\rm l}
:=
  A_{\rm u}
\text{.}
\end{gather*}
Note that this bound also holds, if $c_1 = 0$.  In the case~$c_1 > 0$, the above inequality for $\disc([a_1, b_1, c_1])$ implies that $q_{L_0}(b_1) < a_1 c_1 - c^2 d$.  There are only finitely many vectors~$b_1$ of any bounded length, since $L_0$ is definite.  Consequently, for fixed $a_1$ and $c_1$ only finitely many $b_1$ contribute.  We have thus proved the statement in the case of~$c_1 > 0$.

Consider the case~$c_1 = 0$.  By definition, we have $a_1 \ge 0$ if $[a_1,b_1,0] \succ 0$.  In particular, the above choice of $A_{\rm l}$ is justified also if $c_1 = 0$.  We have already found an upper bound $A_{\rm u}$ on $a_1$ as well.  If $a_1 > 0$, we have the factorization $[a_1, b_1, 0] = m [\td a_1, \td b_1, 0]$ with $m < A_{\rm u}$ and $\disc([\td a_1, \td b_1, 0]) > d$.  Hence $\disc([a_1, b_1, 0]) = -q_{L_0}(b_1) > A_{\rm u}^2 d$.  There are only finitely many $b_1$'s satisfying this bound.  This proves the statement in the case of $c_1 = 0$, $a_1 > 0$.

Consider the case~$a_1 = c_1 = 0$.  Then $b_1 \succ 0$.  Recall the vectors $w_1, \ldots, w_{n - 2} \in L_0 \otimes \QQ$ discussed in Section \ref{sec:preliminaries}.  Since $w_j$ is rational, the set of values $\big| (w_j, b_1) \big|$ with $b_1 \succ 0$ and $1 \le j \le n - 2$ is discrete in $\RR$.  Write $\epsilon$ for the minimal nonzero value of $\big| (w_j, b_1) \big|$, $b_1 \succ 0$.

Given $b_i$, we denote by $1 \le j(b_i) \le n - 2$ the integer satisfying
\begin{gather}
\label{eq:b1_condition}
  (w_{j(b_i)}, b_i) > 0
\quad\text{and}\quad
  (w_k, b_i) = 0
  \;\text{for all $1 \le k < j(b_i)$.}
\end{gather}
Given $1 \le j \le n - 2$, we consider all $b_i \succ 0$ with $j(b_i) = j$.  We will prove by induction on~$j$ that $\big|(w_k, b_i)\big|$ is bounded by some $B_{k\!,\, j} \ge 0$ for all $k$ with $j \le k \le n - 2$.  Given $j$, the induction hypothesis will be that $\big|(w_k, b_i)\big| \le B_{k\!,\, j(b_i)}$ for all $b_i$ with $j(b_i) < j$ and for all $j(b_i) \le k \le n - 2$.  Note that we do not need to treat the case $j = 1$ separately.

Given~$j$, we have $(w_j, b_i) \ge - \max_{1 \le j' < j} B_{j,\, j'}$ for all $b_i$.  Indeed, this follows from the induction hypothesis, if $j(b_i) < j$.  If $j(b_i) = j$, then $(w_j, b_i) > 0$ by definition of $j(b_i)$, and if $j(b_i) > j$, then $(w_j, b_i) = 0$, again by definition of $j(b_i)$.  We conclude that
\begin{gather*}
  (w_j, b_1)
=
  (w_j, b) - \sum_{2 \le i \le l} (w_j, b_i)
\le
  (w_j, b) + (l - 1) \max_{1 \le j' < j} B_{j, j'}
=:
  B_{j, j}
\text{.}
\end{gather*}
Since $0 < (w_j, b_1) \le B_{j, j}$, we can write $b_1 = m \widetilde{b}_1$ with $-q_{L_0}\big(\widetilde{b}_1\big) > d$ and $m \le \epsilon^{-1} B_{j, j}$.  From this, we deduce that $q_{L_0}(b_1) \le - \big( \epsilon^{-1} B_{j, j} \big)^2 d$.  Hence there are only finitely many $b_1$'s that contribute and satisfy~(\ref{eq:b1_condition}) for the given~$j$.  Since there are only finitely many such $b_1$'s, we get the desired bounds $B_{j',\, j}$ for $j < j' \le n - 2$.  This completes the induction.

We easily complete the proof by combining the bounds on $\big| (w_j, b_1) \big|$ and the discreteness of the values~$(w_j, b_1)$ in $\RR$.
\end{proof}

We decompose $\widetilde{\Phi}_F$ as
\begin{gather*}
  \frakA + \frakB + \frakC + \frakD + \frakE
\text{,}
\end{gather*}
where
\begin{align}
\label{eq:frakA_part}
  \frakA &:= \sum_{[a,b,c] > 0}\quad\;\;\;
             - f\big(b,\, \disc([a,b,c]) \big) \sum_{m = 1}^\infty \frac{e^{m [a,b,c]}}{m}
\text{,}
\allowdisplaybreaks
\\[6pt]
  \frakB &:= \sum_{[a,b,c] \not > 0 \,:\, a, c > 0} \!\!\!
             - f\big(b,\, \disc([a,b,c]) \big) \sum_{m = 1}^\infty \frac{e^{m [a,b,c]}}{m}
\text{,}
\allowdisplaybreaks
\\[6pt]
  \frakC &:= \sum_{[a,b,c] \not > 0  \,:\,  a \le 0,\, c > 0} \!\!\!\!\!\!\!\!\!\!
             - f\big(b,\, \disc([a,b,c]) \big) \sum_{m = 1}^\infty \frac{e^{m [a,b,c]}}{m} 
\text{,}
\allowdisplaybreaks
\\[6pt]
  \frakD &:= \sum_{[a,b,c] \,:\, a > 0,\, c = 0} \!\!\!\!\!
             - f\big(b,\, \disc([a,b,c]) \big) \sum_{m = 1}^\infty \frac{e^{m [a,b,c]}}{m} 
\text{,}
\\[6pt]
\allowdisplaybreaks
\label{eq:frakE_part}
  \frakE &:= \sum_{[a,b,c] \,:\, a = c = 0, b \succ 0} \!\!\!\!\!\!\!\!\!
             - f\big(b,\, \disc([a,b,c]) \big) \sum_{m = 1}^\infty \frac{e^{m [a,b,c]}}{m} 
\text{.}
\end{align}
We can further decompose $\frakE = \sum_{j = 1}^{n - 2} \frakE_j$:
\begin{gather*}
  \frakE_j := \sum_{[a,b,c] \,:\, a = c = 0, b \succ 0} \!\!\!\!\!\!\!\!\!
              - f\big(b,\, \disc([a,b,c]) \big) \sum_{m = 1}^\infty \frac{e^{m [a,b,c]}}{m}
\text{,}
\end{gather*}
where the sum only ranges over those $b$ satisfying $(w_j, b) > 0$ and $(w_{j'}, b) = 0$ for all $1 \le j' < j$.  It is not only necessary to make use of this decomposition in order to formulate Algorithm~\ref{alg:borcherdsproduct}, but it will also help us to prove Theorem~\ref{thm:formal_borcherds_convergence}.

Recall that $M$ is the discriminant kernel of $U \oplus L_0$ in $\Gamma$.  In particular, $M$ preserves $\disc([a, b, c])$ and $\ov{b} \in \disc\, L_0$.  Hence $\frakA$ is $M$-invariant.  

The next theorem is used implicitly, when computing Borcherds products, whether or not Borcherds' formula is evaluated only naively.  Nevertheless, no proof seems to be available in the literature.
\begin{theorem}[Formal Borcherds convergence theorem]
\label{thm:formal_borcherds_convergence}
The exponential
\begin{gather}
  {\widetilde \Psi}_F
:=
  W_F \, \exp({\widetilde \Phi}_F)
= 
  W_F
  \sum_{l = 0}^\infty
  \frac{1}{l!}\, {\widetilde \Phi}_F^l
\end{gather}
is well-defined in the sense that, given $[a,b,c]$, the coefficient of $e^{[a,b,c]}$ in ${\widetilde \Phi}_F^l$ is nonzero for only finitely many $l$.

Moreover, for any $[a,b,c]$ the coefficient of $e^{[a,b,c]}$ in ${\widetilde \Phi}_F$ equals the Fourier coefficient of $e\big(a \tau + c \tau' - (b, z)_{L_0} \big)$ in $\Phi_F$.
\end{theorem}
\begin{proof}
The second part follows from the first part, since by Theorem~\ref{thm:borcherds_products}, due to Borcherds, the product expansion of $\Psi_F$ converges locally uniformly and absolutely on $\RR^n + i W$.

The first part of Theorem~\ref{thm:formal_borcherds_convergence} can be shown by analyzing the proof of Lemma~\ref{la:phif_powers}.  Our arguments build upon the above decomposition of~$\widetilde{\Phi}_F$.  Fix $[a, b, c] \in \cQ^+$.  Consider the product
\begin{gather}
\label{eq:formal_borcherds_decomposition}
  \frakA^\alpha \, \frakB^\beta \, \frakC^\gamma \, \frakD^\delta \,
  \prod_{j = 1}^{n - 2} \frakE_j^{\eta_j}
\text{.}
\end{gather}
It suffices to give upper bounds on $\alpha$, $\beta$, $\gamma$, $\delta$, and $\eta_1, \ldots, \eta_{n - 2} \in \NN$ for those products in which $[a,b,c]$ occurs non-trivially.

Since ${\td c} \ge 0$ for all $[{\td a}, {\td b}, {\td c}] \succ 0$, we conclude that $\alpha + \beta + \gamma \le c$.  That is, we have found upper bounds for $\alpha$, $\beta$, and $\gamma$, which we, henceforth, assume to be fixed.  We are reduced to proving that there are upper bounds for $\delta$ and $\eta_1, \ldots, \eta_{n - 2}$.

In analogy to the proof of Lemma~\ref{la:phif_powers}, we can find $A \in \ZZ$ such that ${\td a} \ge A$ for all $[{\td a}, {\td b}, c] \succ 0$ that occur non-trivially in $\frakA^\alpha \, \frakB^\beta \, \frakC^\gamma$.  We deduce that $\delta \le a - A$.

Henceforth, we assume that $\delta$ is fixed, and we will prove that there are only finitely many $\eta_1, \ldots, \eta_{n - 2}$ leading to a nonzero coefficient of $[a, b, c]$ in the product~\eqref{eq:formal_borcherds_decomposition}.  We use induction on the index~$j$ of $\eta_j$.  Given~$j$, assume that we have bounds on $\eta_{j'}$ for all $1 \le j' < j$.  For simplicity fix $\eta_{j'}$ for these $j'$.  Inspecting the proof of Lemma~\ref{la:phif_powers}, we find a lower bound on $(w_j, {\td b})$ that holds for all $[a, {\td b}, c]$ that satisfy $(w_{j'}, {\td b}) = (w_{j'}, b)$ for all $1 \le j' < j$ and that occur non-trivially in
\begin{gather*}
  \frakA^\alpha \, \frakB^\beta \, \frakC^\gamma \, \frakD^\delta \, \prod_{1 \le j' < j} \frakE_{j'}^{\eta_{j'}}
\text{.}
\end{gather*}
For all $[{\td a}, {\td b}, {\td c}]$ that occur non-trivially in $\frakE_{j'}$, $j < j' \le n - 2$ we have $(w_{j'}, {\td b}) = 0$.  Since, in addition, $(w_j, {\td b}) > \epsilon$ (with $\epsilon$ taken from the proof of Lemma~\ref{la:phif_powers}) for all $[0, {\td b}, 0]$ that occur non-trivially in $\frakE_j$, we obtain a bound on $\eta_j$.  This completes the proof.
\end{proof}

Choose functions
\begin{gather*}
  w_{j, \max} : \NN \rightarrow \RR
\text{,}\quad
  w_{j, \max}(B)
\ge
  \max_{b \in L_0,\, q_{L_0}(b) < B} \big| (w_j, b) \big|
\text{.}
\end{gather*}
The optimal choices for $w_{j, \max}$ are invariants of the lattice~$L_0$.  In practice, it seems sufficient to give an estimate computed by embedding $L_0$ into a lattice that is diagonal with respect to $w_1, \ldots, w_{n - 2}$.  We write $\xi_\alpha$ for the Fourier coefficients of $\Xi_\alpha \in \ov{\QQ}\llbracket \cQ \rrbracket$.

\begin{algorithm}
\label{alg:borcherdsproduct}
Assume that $F$ is a vector valued modular form of weight $(n - 2) \slashdiv 2$ and type $\rho_{L}$.  Suppose that $f(b, d) = 0$ if $d \le d_{\min}$.  Given $0 < B \in \ZZ$, set $a_{\rm neg} = - (B - c_W - 1)^2 \, d_{\min}$.  Then we need the Fourier coefficient $f(b, d)$ for $d \le D := \max \{0, (B - 1 - a_W + a_{\rm neg}) (B - 1 - c_W) \}$ in order to compute all Fourier coefficients $\widetilde{\phi}_F([a, b, c])$ of ${\widetilde \Phi}_F$ with $a, c < B$ using the following algorithm.
\begin{enumerate}[1:]
\item \label{it:truncate_frakA_frakB}
      Truncate $\frakA$ and $\frakB$ to indices $[a, b, c]$ with $a < B - a_W + a_{\rm neg}$ and $c < B - c_W$.  Truncate $\frakC$ to indices $[a, b, c]$ with $c < B - c_W$.
\vspace{0.5ex}

\item For all $0 \le \alpha < B - c_W$, set 
  \begin{gather*}
    \Xi_{\alpha}
  \leftarrow
    \sum \,
    \frac{1}{\alpha! \beta! \gamma!} \, \frakB^\beta\, \frakC^\gamma
  \text{,}
  \end{gather*}
  where the sum ranges over $0 \le \beta, \gamma < B - c_W - \alpha$.

\item Set $a_{\min} \leftarrow \min \big\{a \,:\; \exists \alpha, [a,b,c] \text{ such that } \xi_\alpha \big( [a,b,c] \big) \ne 0 \big\}$.
\vspace{0.5ex}

\item \label{it:delta_truncation}
      Truncate $\frakD$ to indices $[a, b, c]$ with $a < B - a_W - a_{\min}$.
\vspace{0.5ex}

\item For all $0 \le \alpha < B - c_W$, set 
  \begin{gather*}
    \Xi_\alpha \leftarrow \Xi_\alpha \cdot \sum \,
                                         \frac{1}{\delta!} \, \frakD^\delta
  \text{,}
  \end{gather*}
  where the sum ranges over~$0 \le \delta < B - a_W - a_{\min}$.
\vspace{0.5ex}

\item \label{it:b_component_iteration}
      For $1 \le j \le n - 2$ process the following loop.

\begin{enumerate}[{\ref{it:b_component_iteration}.1}]

\item Set
  \begin{align*}
    a_{\min}
  \leftarrow {} &
     \min \big\{a \,:\; \exists \alpha, [a,b,c] \text{ such that } \xi_\alpha \big( [a,b,c] \big) \ne 0 \big\}
  \text{,}
  \\
    b_{\min}
  \leftarrow {} &
     \min \big\{(w_j, b) \,:\; \exists \alpha, [a,b,c] \text{ such that } \xi_\alpha \big( [a,b,c] \big) \ne 0 \big\}
  \text{,}
  \\
    c_{\min}
  \leftarrow {} &
     \min \big\{c \,:\; \exists \alpha, [a,b,c] \text{ such that } \xi_\alpha \big( [a,b,c] \big) \ne 0 \big\}
  \text{,}
  \\
    b_{\rm tru}
  \leftarrow {} &
    b_{\min} + (w_j, b_W)
    + w_{j, \max} \big((B - 1)^2 + 1\big)
  \\
  {} &
    + w_{j, \max}\big( (B - 1 - a_W - a_{\min}) (B - 1 - c_W - c_{\min}) + 1 \big)
  \text{.}
  \end{align*}

\item Truncate $\frakE_j$ to indices $[a, b, c]$ satisfying $(w_j, b) \le b_{\rm tru}$.
\vspace{0.5ex}

\item \label{it:mult_delta}
      For all $0 \le \alpha < B - c_W$, set 
  \begin{gather*}
    \Xi_\alpha \leftarrow \Xi_\alpha \cdot \sum_{0 \le \eta_j \le b_{\rm tru}}
                                         \frac{1}{\eta_j!} \, \frakE_j^{\eta_j}
  \text{.}
  \end{gather*}
\vspace{0.5ex}
\end{enumerate}

\item \label{it:finial_computation}
      Compute 
  \begin{gather*}
    \cF\cE
  \leftarrow
    e^{[a_W, b_W, c_W]}
    \sum_{0 \le \alpha < B - c_W}
      \Xi_\alpha \, \frakA^\alpha
n  \text{,}
  \end{gather*}
  and truncate it to indices $[a, b, c]$ satisfying $a, c < B$.  
\end{enumerate}
Then $\cF\cE$ represents the truncated Fourier expansion of $\widetilde{\Phi}_F$. 
\end{algorithm}
\begin{proof}[Proof of correctness]
In order to see that the given Fourier coefficients $f(b, d)$ suffice to compute the Fourier expansion~$\cF\cE$, note that the $f(b, d)$ with $d > 0$ only contribute to the expressions of $\frakA$.  Then the statement follows from inspection of Step~\ref{it:truncate_frakA_frakB}.

Correctness follows when making the bounds used in the proof of Lemma~\ref{la:phif_powers} and Theorem~\ref{thm:formal_borcherds_convergence} explicit.  Since the important arguments are already given there, we content ourselves with sketching how to do this.

The factor $e^{[a_W, b_W, c_W]}$ is taken into consideration only in the last step.  Hence we have to prove that
\begin{gather}
\label{eq:algorithm:proof}
  \sum_{0 \le \alpha < B - c_W}
    \Xi_\alpha \, \frakA^\alpha
\end{gather}
represents the Fourier expansion $\cF\cE'$ of $\widetilde{\Psi}_F \slashdiv e^{[a_W, b_W, c_W]}$ for all $[a, b, c]$ with $a < B - a_W$ and $c < B - c_W$.  Only terms of index $[a, b, c]$ with $c < B - c_W$ contribute to $\cF\cE'$, since $c \ge 0$, if $[a, b, c]$ occurs non-trivially in $\frakA$, $\frakB$, or $\frakC$.  Hence it is correct to truncate $\frakA$, $\frakB$, and $\frakC$ to $c < B - c_W$.  In the same way, we can see that we only have to consider powers of $\frakA$ with exponent less than $B - c_W$.  Since $\cF\cE = \sum_\alpha \frakA^\alpha \, \Xi_\alpha$ and $c \ge 1$ for all $[a, b, c]$ which occur non-trivially in $\frakA$, we only have to consider powers of $\frakB$ and $\frakC$ with exponents less than $B - c_W - \alpha$ when computing $\Xi_\alpha$. 

One has to take $\frakC$ into account in order to see that truncating $\frakA$ and $\frakB$ to $a < B - a_W - a_{\rm neg}$ is correct.  The minimal value of $a$ for $[a, b, c]$ that occurs non-trivially in any power of $\frakC$ and contributes to $\cF\cE'$ is bounded by $a_{\rm neg}$.  Since $a > 0$ for terms that occur non-trivially in $\frakA$ or $\frakB$, also this truncation is correct.  Hence when Step~\ref{it:delta_truncation} is processed, then contributions of $\frakB$ and $\frakC$ to $\cF\cE$ have been computed correctly.

When processing Step~\ref{it:delta_truncation}, the minimal value of $a$ for indices $[a, b, c]$ which occur in any $\frakA^\alpha \Xi_\alpha$, $0 \le \alpha < B - c_W$ is $a_{\min}$.  Since $a > 0$ for any $[a,b,0]$ that occurs non-trivially in $\frakD$, we have to take into consideration powers of $\frakD$ with exponent less than $B - a_W - a_{\min}$.  The terms of $\frakD$ that contribute satisfy $a < B - a_W - a_{\min}$.

The contributions of $\frakE_j$ can be analyzed along the same lines.  We content ourselves with explaining the terms that occur in the expression for $b_{\rm tru}$.  The first term is the minimal value of $(w_j, b)$ in $\Xi_\alpha$.  The second term takes the Weyl vector into account.  The third term in the expression for $b_{\rm tru}$ corresponds to the maximal value of~$|(w_j, b)|$ for those $[a, b, c]$ that occur non-trivially in $\cF\cE$.  The fourth term corresponds to the maximal value of~$|(w_j, b)|$ for those $[a, b, c]$ that occur non-trivially in any $\frakA^\alpha$.
\end{proof}

\begin{remark}
Considering each step separately, one easily sees that Algorithm~\ref{alg:borcherdsproduct} has polynomial runtime in $B$.  See the introduction for a discussion of evaluation based on naive multiplication.
\end{remark}


\section{Borcherds products for the Hermitian modular group}
\label{sec:hermitianborcherdsproducts}

The authors provide an implementation of Algorithm~\ref{alg:borcherdsproduct} for Hermitian modular forms over the imaginary quadratic fields $\QQ(\sqrt{D})$ with $D$ a fundamental discriminant.  It is written in Sage~\cite{sage}, and builds on a framework presented in~\cite{Ra11}, which provides a model for general Fourier expansions.  For $D \ne -3$ the framework currently provides only limited support, resulting in restricted functionality of our implementation.  Both the implementation of the discussed algorithm and the Fourier expansion framework are part of the last named author's branch of Purple Sage~\cite{purplesage, psagerepository}.  In Section~\ref{ssec:using_implementation}, we will explain how to install and use them.

We briefly describe the model for the Fourier expansion that we have chosen.  We follow the notation introduced by Braun~\cite{Br49, Br50, Br51}.  Hermitian modular forms in her sense can also be described as orthogonal modular forms for~$\Orth{2}{4}$.  The lattice $L$ that corresponds to the case of Hermitian modular forms is $L = U \oplus U \oplus (-1) L_0$, where $L_0$ is the integers in $\QQ(\sqrt{D})$ together with the quadratic form $q_{L_0}(a) = a \ov{a} \in \ZZ$.  In the case of $D = -3$, which we will deal with throughout this section, we can choose $L_0$ to be a lattice with Gram matrix $\left(\begin{smallmatrix} 2 & 1 \\ 1 & 2 \end{smallmatrix}\right)$.  Hermitian modular forms are functions of four complex variables, written as $Z \in \Mat{2}{\CC}$ satisfying \mbox{$\big(Z - \ov{Z}^\tr\big) \slashdiv i > 0$}.  Denote the entries of $Z$ by $z_{kl}$, $1 \le k,l \le 2$.  Then $z_{12} + z_{21} + j (z_{12} - \ov{z_{21}}) \in \CC \otimes \CC_j$ corresponds to $z \in L_0 \otimes \CC_j$.  Here, $\CC_j$ denotes the field of complex numbers $x + j y$, $j^2 = -1$, $x, y \in \RR$, and $L_0$ is identified with $\CC$ as a quadratic space via the basis $1,\, (D + \sqrt{D})\slashdiv 2$.  For more details on the correspondence between orthogonal modular forms for $L$ and Hermitian modular forms and for proofs of what we state in the next two paragraphs, the reader is referred to~\cite{De01}.  We have, however, chosen notation so that notions in the former setting can be easily translated to notions in the latter.  In particular, the Fourier terms $\exp(2 \pi i\, \tr([a, b, c] Z))$ in the Hermitian setting correspond to the Fourier terms of index $[a, b, c]$ in the previous section.

The Fourier expansion of an Hermitian modular form without character is indexed by Hermitian matrices $[a, b, c] := \left(\begin{smallmatrix}a & b \\ \ov{b} & c\end{smallmatrix}\right)$, where $a, c \in 2 \ZZ$ and $b \in \cO^\#$.  By $\cO$ we denote the ring of integers in $\QQ\big( \sqrt{D} \big)$, and $\cO^\#$ is the inverse different of $\QQ\big( \sqrt{D} \big)$.  We store such matrices as quadruples $(a, b_1, b_2, c)$, where $b$ is represented by a pair $(b_1, b_2)$ of coordinates with respect to the basis $1 \slashdiv \sqrt{D},\, (1 + \sqrt{D}) \slashdiv 2$ of $\cO^\#$.  Fourier expansions of Hermitian modular forms are almost invariant under $\GL{2}(\cO)$.  Given such a Fourier expansion $\sum a(T) \exp(2 \pi i\, \tr(T Z))$, we have $a(T) = \det(\ov{U})^{k} a\big( \ov{U}^\tr T U \big)$ for all $U \in \GL{2}(\cO)$.  We have chosen a Weyl chamber for~$L$ such that the condition $b \succ 0$, which was described in Section~\ref{ssec:borcherdsproducts}, is equivalent to \mbox{$b_2 < 0 \vee (b_2 = 0 \wedge b_1 < 0)$}.  The associated Weyl chamber, as a subset of $(U \oplus (-1)\cO^\#) \otimes \RR$, where $\cO^\#$ has basis $\frac{1}{\sqrt{D}}$, $\frac{1 + \sqrt{D}}{2}$ as above, contains the set
\begin{gather*}
  \big\{
  \big( 1, x, x^2 (-D + \tfrac{D^2 - D}{2} x), x^2 (2 - D x) \big)
  \,:\, 0 < x < \epsilon
  \big\}
  \text{,}
\end{gather*}
provided that $\epsilon$ is sufficiently small.  In order to check that the above is indeed always contained in some Weyl chamber, adjust the arguments on page~105 of~\cite{De01}.

The implementation follows Algorithm~\ref{alg:borcherdsproduct} closely, except that unnecessary parts of $\Xi_\alpha$ are truncated after each step.  Computing powers of $\frakA$ is the bottleneck of our implementation.  As mentioned in the previous section, $\frakA$ is invariant under the action of $M$, and we have $M \cong \GL{2}(\cO)$.  We use this to speed up the computations.  The multiplication of two $\GL{2}(\cO)$-invariant Fourier expansions is performed by the Fourier expansion framework.  Its most time consuming parts are implemented in Cython.

The implementation of the last step of Algorithm~\ref{alg:borcherdsproduct} also relies on $\GL{2}(\cO)$ invariant Fourier expansions.  This invariance allows for a significant reduction of the number of coefficients that we need to calculate.  That is, we write the product
\begin{gather*}
  \sum_{[a, b, c]} g([a, b, c]) \exp(2 \pi i \, \tr([a, b, c] Z))
  \cdot
  \sum_{[a, b, c]} h([a, b, c]) \exp(2 \pi i \, \tr([a, b, c] Z))
\end{gather*}
as
\begin{gather*}
  \sum_{[a, b, c]} \Big( \sum_{[a, b, c] = [a_1, b_1, c_1] + [a_2, b_2, c_2]} \hspace{-2.5em} g([a_1, b_1, c_1]) \, h([a_2, b_2, c_2]) \Big) \exp(2 \pi i \, \tr([a, b, c] Z))
\text{.}
\end{gather*}
Now, the inner expression needs to be calculated only for few $[a, b, c]$.  For more details on the implementation, the reader it referred to comments in the source code.

\subsection{Runtimes and tests}
As a first demonstration of our implementation, we illustrate how it compares to naive multiplication of all factors in~(\ref{eq:borcherds_psif}).  More precisely, each factor is truncated and the resulting polynomials are multiplied using FLINT.  In addition, we have applied intermediate truncation to avoid too large polynomials.  Bounds on the Fourier indices $[a, b, c]$, that we needed in order to perform this truncation, can essentially be found in Section~\ref{sec:algorithm}.  We have imposed a simplified version of these bounds, that does not depend on the decomposition of~(\ref{eq:borcherds_psif}) into~$\frakA$, $\frakB$, etc.  The precise bounds can be found in the file {\tt phi\_45\_naive\_computation.sage} at~\cite{RaHomepage}.

The Hermitian modular form~$\Phi_{45}$ occurred in~\cite{De01} and~\cite{DK03} as a generator of the graded ring of Hermitian modular forms over $\QQ(\sqrt{-3})$.  Table~\ref{tab:runtimes} shows times needed to compute the Fourier expansion $\sum_{[a, b, c]} \Phi_{45}([a, b, c]) \, \exp\big( 2 \pi i\, \tr([a,b,c] Z) \big)$ of~$\Phi_{45}$.  The columns correspond to precisions: We have computed the Fourier coefficients $\phi_{45}([a, b, c])$ for all $a, c < B$ and all~$b$.  The first row contains a list of times consumed by an implementation based on naive multiplication, the second row contains a list of times consumed by our implementation.

In order to test our implementation for correctness, we have compared the results of both computations for precision~$8$.  Some of the coefficients are given in Table~\ref{tab:phi45_coefficients}.

\begin{table}[h]
\begin{tabular}{l@{\hspace{1.5em}}lllll}
\toprule
$B$ & $\;5$ & $\;6$ & $\;7$ & $\;8$ & $\; 9$ \\
\midrule
Number of computed coefficients & $1011$      & $2353$
                                  & $4627$      & $8301$
                                  & $13765$ \\[4pt]
Multiplication                    & $0.05 \,{\rm s}$ &  $4.9 \,{\rm s}$
                                 & $ 290.0 \,{\rm s}$ & $142.6\, {\rm min}$ 
                                 & \;--- \\
Algorithm~\ref{alg:borcherdsproduct}  & $0.20 \,{\rm s}$ &  $0.7 \,{\rm s}$ 
                                 & $ \hphantom{00}3.8 \,{\rm s}$ & $\hphantom{0}19.7\, {\rm s}$ 
                                 & $93.9 {\rm s}$ \\
\bottomrule \\
\end{tabular}
\caption{Time needed to compute the coefficients of $\Phi_{45}$ for indices with diagonal entries bounded by $B$.}
\label{tab:runtimes}
\end{table}

\begin{table}[h]
\begin{tabular}{llllll}
\toprule
$[a, b_1, b_2, c]$ & $[3, 3, 2, 4]$ & $[3, 3, 2, 5]$ & $[3, 3, 2, 6]$
                   & $[4, 3, 2, 5]$ & $[4, 3, 2, 6]$ \\[2pt]
$\phi_{45}([a,b,c])$ & $-1$ & $88$ & $-3740$
                    & $95931$ & $-720940$ \\
\midrule
$[a, b_1, b_2, c]$ &  $[4, 4, 3, 5]$ & $[4, 4, 3, 6]$ & $[5, 3, 2, 6]$ &
                  $[5, 4, 3, 6]$ & $[5, 5, 4, 6]$ \\[2pt]
$\phi_{45}([a,b,c])$ & $16038$ & $681615$ & $-835953624$
                    & $62772732$ & $-47271276$ \\
\bottomrule \\
\end{tabular}
\caption{Fourier coefficients of $\Phi_{45}$, where $b = b_1 \slashdiv \sqrt{D} + b_2 (1 + \sqrt{D}) \slashdiv 2$ }
\label{tab:phi45_coefficients}
\end{table}

\subsection{Input data for the implementation}
\label{ssec:vector_valued_format}

We now explain the format of Fourier expansions of vector valued modular forms that can be passed to our implementation.  Such a Fourier expansion is represented by a dictionary~{\tt f} whose keys are tuples and whose values are dictionaries, which we call the inner dictionaries of~{\tt f}.  The keys of~{\tt f} represent vectors in~$\cO^\# \slashdiv \cO = \disc\, L$ with respect to the basis $1/\sqrt{D},\, (1 + \sqrt{D}) \slashdiv 2$ of~$\cO^\#$.  A unique representative can be obtained as follows.
\begin{lstlisting}
D = -3
prec = HermitianModularFormD2Filter_diagonal(2, D)
factory = HermitianModularFormD2Factory(prec)
factory._reduce_vector_valued_index((2,3))
\end{lstlisting}
This gives the tuple $(-1, 0)$.  That is, we have
\begin{gather*}
  2 \cdot \frac{1}{\sqrt{-3}} + 3 \cdot \frac{1 + \sqrt{-3}}{2}
=
  -1 \cdot
  \frac{1}{\sqrt{-3}}
\;\in
  \disc(L)
\text{.}
\end{gather*}
We say that a tuple is reduced if it is reproduced by {\tt \_reduce\_vector\_valued\_index}.  The keys of the input dictionary~{\tt f} must by reduced.  In the case $D = -3$ the reduced tuples are $(0,0)$, $(1,0)$, and $(-1,0)$.

The inner dictionaries of~{\tt f} have rational numbers as keys and values.  The keys correspond to the exponent of $q$ in the Fourier expansion of the corresponding component.  The values are the associated Fourier coefficients.
\begin{example}
The truncated Fourier expansion of a vector valued modular form
\begin{gather*}
  F_{0} = 1 + 2\, q + O(q^2)
\text{,}
\quad
  F_{\frac{\pm 1}{\sqrt{-3}}} = 4\, q^{\frac{1}{3}} + O(q^2)
\end{gather*}
can be represented as follows.
\begin{lstlisting}[basicstyle=\ttfamily\small]
  {  (0,0) : { 0 : 1, 1 : 2 },  (1,0) : { 1/3 : 4 },
    (-1,0) : { 1/3 : 4 } }
\end{lstlisting}
\end{example}

In some special cases vector valued modular forms of type~$\rho_L$ can be computed using Eisenstein series and theta series.  However, this approach does not always work.  In~\cite{Ra12}, a method is provided that allows us to compute vector valued modular forms for all~$L$.  The input data used in Section~\ref{ssec:using_implementation} was obtained this way.

\subsection{Using the implementation}
\label{ssec:using_implementation}

One can apply explicit computations of Fourier expansions to prove relations between various kinds of modular forms.  For example, in~\cite{GK10} the restrictions of quaternionic modular forms to the Hermitian and Siegel upper half space were investigated.  We shall establish a product expansion for a certain elliptic modular form.  The case that we consider has no particular importance, but it was chosen for reasons of simplicity and clearness.  It is, however, similar to computations that are necessary in string theory, where one wants to derive equalities between Borcherds products and Maass lifts~\cite{DG07}.

Install Sage~\cite{sage} in such a way that you can modify it.  Using git~\cite{githomepage}, clone the author's Purple Sage repository at~\cite{psagerepository}.
\begin{lstlisting}[basicstyle=\ttfamily\small]
git clone git://github.com/martinra/psage.git
cd psage
git fetch origin paper_computing_borcherds_products:borcherds
git checkout borcherds
sage -sh
./build_ext
\end{lstlisting}
Link the downloaded psage library to the site~packages folder of Sage's Python.  If, for example, both Sage and Purple~Sage are in the same folder, go there and type:
\begin{lstlisting}[basicstyle=\ttfamily\small]
PSAGE=`readlink -f psage/psage`
ln -s $PSAGE sage/local/lib/python/site-packages/
\end{lstlisting}
Start Sage, either in the terminal or as a notebook.  The code that proves the final equality~(\ref{eq:example_hermitian_equality}) can be found at~\cite{RaHomepage}.  One can either load the corresponding file by
\begin{lstlisting}[basicstyle=\ttfamily\small]
load borcherds_example.sage
\end{lstlisting}
or copy the code line by line.  We reproduce parts of this code, which illustrate how to invoke our implementation of Algorithm~\ref{alg:borcherdsproduct}.  The input data is a vector valued modular form~{\tt F}.
\begin{lstlisting}[basicstyle=\ttfamily\small]
F = load("borcherds_example_input_data.sobj")
B = 11; D = -3;
prec = HermitianModularFormD2Filter_diagonal(B, D)
phi = borcherds_product__by_logarithm(F, prec)
\end{lstlisting}
We start by loading a file that contains the Fourier expansion~$f$ of a vector valued elliptic modular form.  In the third line, we define a filter in the sense of~\cite{Ra11}.  It contains all Hermitian matrices~$[a, b, c]$ satisfying $a, c < B$.  In the last line, we invoke our implementation.  This way, we obtain a Fourier expansion~{\tt phi} of the corresponding Borcherds product.  The resulting list of Fourier coefficients for reduced indices~$[a, b, c]$ can be found at~\cite{RaHomepage}

In order to demonstrate how to make further use of the result, we restrict it to the Poincar\'e upper half plane.
\begin{lstlisting}[basicstyle=\ttfamily\small]
prec_nr = phi.parent().monoid().filter(B)
mf_expansion = dict( enumerate(2*B * [0]) )
ch = phi.parent().characters().one_element()
for (a, b1, b2, c) in prec_nr :
  mf_expansion[a + c] += phi[(ch, (a, b1, b2, c))]
mfs = ModularForms(1, F[(0,0)][0] / 2)
mfs(PowerSeriesRing(QQ, 'q')(mf_expansion).add_bigoh(B))
\end{lstlisting}
In the first line, we obtain a filter that allows us to access all indices of a truncated Fourier expansion.  Since we have set $B = 11$, this is a list of all positive definite Hermitian matrices~$[a, b, c]$ as above satisfying $a, c < 11$.  In the second to fifth line, we restrict the Borcherds product to the Poincar\'e upper half plane.  The corresponding embedding of half spaces is
\begin{gather*}
  \tau \mapsto \begin{pmatrix} \tau & 0 \\ 0 & \tau \end{pmatrix}
\text{.}
\end{gather*}
The reader is referred to~\cite{FH00} for a general treatment of modular embeddings in the context of orthogonal modular forms.  In the last line, we construct an elliptic modular form from the dictionary that we have obtained before.

We have thus found an explicit expression for
\begin{gather}
\label{eq:example_hermitian_equality}
  W_F \prod_{[a,b,c] \succ 0} \big(1 - q^{a + c} \big)^{f( b, \disc([a, b, c]) )}
=
  - \Delta(\tau)^9
\text{,}
\end{gather}
where $\Delta$ is the unique normalized cuspform of weight~$12$.  The form~$F$ is uniquely determined by the Fourier expansion
\begin{gather*}
  F_0(\tau) = q^{-2} + O(q^0)
\text{,}\quad
  F_{\frac{\pm 1}{\sqrt{-3}}}(\tau) = O(q^0)
\text{.}
\end{gather*}

The above computation is time consuming on usual computers.  The Borcherds product $\Phi_{45}$ is a less interesting but more accessible example.  For reasons of symmetry, the restriction of $\Phi_{45}$ to $\left(\begin{smallmatrix}\tau & 0 \\ 0 & \tau \end{smallmatrix}\right)$ vanishes, and this can be checked by a direct computation.  Along the way, the reader will compute the coefficients of $\Phi_{45}$ which are given in Table~\ref{tab:phi45_coefficients}.
\begin{lstlisting}[basicstyle=\ttfamily\small]
F = load("phi_45_input_data.sobj")
B = 7; D = -3
prec = HermitianModularFormD2Filter_diagonal(B, D)
phi = borcherds_product__by_logarithm(F, prec)
\end{lstlisting}
Now one can proceed as above to obtain the restriction of $\Phi_{45}$.


\bibliographystyle{amsalpha}
\bibliography{bibliography}

\end{document}